\documentclass{amsart}

\usepackage{amssymb}

\newcommand{\NN}{{\mathbb N}}

\numberwithin{equation}{section}

\newtheorem{theo}{Theorem}    
\newtheorem{prop}[theo]{Proposition}  
\newtheorem{coro}[theo]{Corollary}  
\newtheorem{lemma}[theo]{Lemma} 
 
\newtheorem{defi}[theo]{Definition}

\begin{document}

\title[Cobham's theorem for substitutions]{Cobham's theorem for substitutions}
\author{Fabien Durand}
\address[F.D.]{\newline
Universit\'e de Picardie Jules Verne\newline
Laboratoire Ami\'enois de Math\'ematiques Fondamentales et Appliqu\'ees\newline
CNRS-UMR 6140\newline
33 rue Saint Leu\newline
80039 Amiens Cedex \newline
France.}
\email{fabien.durand@u-picardie.fr}

\begin{abstract}
The seminal theorem of Cobham has given rise during the last 40 years to a lot of works around non-standard numeration systems and has been extended to many contexts.
In this paper, as a result of fifteen years of improvements, we obtain a complete and general version for the so-called substitutive sequences.

Let $\alpha$ and $\beta$ be two multiplicatively independent Perron numbers.
Then, a sequence $x\in A^\mathbb{N}$, where $A$ is a finite alphabet, is both $\alpha$-substitutive and $\beta$-substitutive if and only if $x$ is ultimately periodic.
\end{abstract}

\maketitle

\section{Introduction}
The seminal theorem of Cobham has given rise during the last 40 years to a lot of works around non-standard numeration systems and has been extended to many contexts.
The original Cobham's theorem is concerned with integer base numeration system. 
In this paper, as a result of fifteen years of improvements, we obtain a complete and general version for the so-called substitutive sequences.

A set $E\subset \NN$ is {\it $p$-recognizable} for some
$p\in\NN\setminus\{0,1\}$, if the language consisting of the $p$-ary
expansions of the elements in $E$ is recognizable by a finite
automaton.
It is obvious to see that $E$ is recognizable if and only if it is $p^k$-recognizable.
In 1969, A.~Cobham obtained the following remarkable theorem.

\medskip

{\bf Cobham's theorem.}
\cite{Cobham:1969} 
{\it Let $p,q\ge 2$ be two multiplicatively independent
  integers (i.e., $p^k \not = q^\ell$ for all integers $k,\ell>0$). A
  set $E\subset\NN$ is both $p$-recognizable and $q$-recognizable if
  and only if $E$ is a finite union of arithmetic progressions.}

\medskip 

It is interesting to recall what S. Eilenberg wrote in his book \cite{Eilenberg:1974}: {\em The proof
  is correct, long and hard. It is a challenge to find a more
  reasonable proof of this fine theorem}.  
To this aim G. Hansel proposed a
simpler presentation in \cite{Hansel:1982}, also one can see 
\cite{Perrin:1990} or the dedicated chapter in
\cite{Allouche&Shallit:2003} for an expository presentation where a mistake is corrected in \cite{Rigo&Waxweiler:2006}. 

In \cite{Cobham:1972}, Cobham make precise the structure of these $p$-recognizable sets: they are exactly the images by letter-to-letter morphisms of constant-length substitution fixed points.
He also defined the notion of $p$-automatic sequences: The $n$-th term of the sequence is a mapping of the last reached state of the automaton when its input is the digits of $n$ is some given base $p$ numeration system.
Clearly $E\subset \mathbb{N}$ is $p$-recognizable if and only if its characteristic sequence is $p$-automatic.
Automata provide a nice and easy description of $p$-recognizable sets whereas substitutions afford an algorithm to produce such sets.
From there a lot of other characterizations have been given, 
the first and major being in terms of:

\begin{enumerate}
\item
$p$-definable sets (see \cite{Bruyere&Hansel&Michaux&Villemaire:1994} for a survey);
\item
$p$-kernel \cite{Eilenberg:1974};
\item
(when $p$ is prime) algebraic series over $\mathbb{F}_p (X)$ \cite{Christol:1979,Christol&Kamae&MendesFrance&Rauzy:1980}.
\end{enumerate}
  
This opened a wide range of further works we briefly describe below.

The Cobham theorem has been generalized to recognizable subsets of $\mathbb{N}^d$ by A.~L.~Semenov in \cite{Semenov:1977} for base $p$ numeration systems (also called later {\em standard numeration systems}) to give the so-called Cobham--Semenov theorem (see \cite{Bruyere&Hansel&Michaux&Villemaire:1994} for a nice survey).
Then the efforts were concentrated to simplify Cobham--Semenov's theorem and to generalize it to non-standard numeration systems given by linear recurrence relations like the Fibonacci one.
Alternative and satisfactory proofs (in terms of simplification) have been proposed, among them a very interesting logical proof due to C. Michaux and R. Villemaire \cite{Michaux&Villemaire:1993,Michaux&Villemaire:1996} (see also \cite{Bes:1997,Muchnik:2003}), using $p$-definable sets in the formalism of first order logic into some arithmetic models like the Presburger arithmetic $\langle\mathbb{N},+\rangle$ (see \cite{Bes:2001} for a survey on these methods).
Another recent proof makes use of ergodic measures \cite{Durand:2008} and the fact that $p$-recognizable subsets of $\mathbb{N}^ d$ are characterized by multidimensional substitutions \cite{Cerny&Gruska:1986, Salon:1987}.
This last characterization is an extension of Cobham's result of 1972.

The first result obtained for non-standard numeration systems is due to S. Fabre \cite{Fabre:1994}.
He considered subsets of $\mathbb{N}$ that are both $p$-recognizable and $U$-recognizable where $U$ is a non-standard numeration system associated with some restricted class of Pisot numbers.
Then, V. Bruy\`ere and F. Point \cite{Point&Bruyere:1997} proved such a result for subsets of  $\mathbb{N}^d$ under less restrictive (and more natural) assumptions on the involved Pisot numbers.
Later A. B\`es \cite{Bes:2000} succeeded in generalizing Cobham--Semenov's theorem for
subsets of $\mathbb{N}^d$ recognizable by automata in two non-standard numeration systems which are
associated with the minimal polynomials of multiplicatively independent Pisot numbers. 
Up to now it is the best generalization obtained for $d\geq 2$.

Prior to this result, for $d=1$, Cobham's theorem was extended to a much wider class of non-standard numeration systems in \cite{Durand:1998c} where no ``Pisot conditions'' are needed, and
later in \cite{Durand&Rigo:2009} to abstract numeration systems (defined in \cite{Rigo:2000}). 
This last result includes all previously known such result in dimension 1.
It is important to notice that the proofs in  \cite{Durand:1998c,Durand&Rigo:2009} used substitution fixed points while the papers \cite{Bes:2000,Bruyere&Hansel&Michaux&Villemaire:1994,Michaux&Villemaire:1993,Michaux&Villemaire:1996,Muchnik:2003,Point&Bruyere:1997} used a first order logic approach.

With a substitution is associated an integer square with non-negative entries. 
It is well-known (see \cite{Horn&Johnson:1990} for
instance) that such a matrix has a real eigenvalue $\alpha$ which is
greater than or equal to the modulus of all other eigenvalues. 
It is usually called
the dominant eigenvalue of $M$. 
This allows us to define the notion of $\alpha$-substitutive sequences. 
From the characterization given in \cite{Cobham:1972} it is easy to deduce that the characteristic sequences of $p$-recognizable subsets of $\mathbb{N}$ are $p$-substitutive sequences. 
This suggested to G. Hansel the following result (see \cite{Allouche&MendesFrance:1995}) that is the main result of this paper.

\begin{theo}
\label{theo:main}
Let $\alpha$ and $\beta$ be two multiplicatively independent Perron numbers. Let $A$ be a finite alphabet and $x$ be a sequence of $A^{\mathbb{N}}$. Then, 
$x$ is both $\alpha$-substitutive and $\beta$-substitutive
if and only if
$x$ is ultimately periodic.
\end{theo}

Partial answers have been given in \cite{Fabre:1994,Durand:1998b,Durand:2002b} with conditions on the substitutions.

Observe that the main result in \cite{Durand:1998c} on numeration systems is a consequence of the Cobham theorem for primitive substitutions established in \cite{Durand:1998b} thanks to a result of Fabre in \cite{Fabre:1995} that characterizes the characteristic sequences of recognizable sets of integers in non-standard numeration systems in terms of fixed points of substitutions.

Let us mention further generalizations of Cobham's theorem.

We could weaken the assumption that a sequence is $p$ and $q$-automatic 
assuming that a $p$-automatic and a $q$-automatic sequence share the same language (the set of finite words occurring in the sequence), and, asking
 if a Cobham theorem type result still holds in this context.
It holds as proven in \cite{Fagnot:1997}.
It was generalized to primitive substitutions in \cite{Durand:1998c}.
Translated to the framework of dynamical systems this result means that if two subshifts, one being generated by a $p$-automatic sequence and the other by a $q$-automatic one (with $p$ and $q$ multiplicatively independent), have a common topological factor then it contains a unique minimal (which is also uniquely ergodic) subshift, moreover it is periodic. 
This point of view is developped in \cite{Durand:2002b} and provides a new proof of Cobham's type theorem.
It uses the ergodic measures of the subshifts and the values they take on cylinder sets.
Moreover this way to tackle the problem also works in higher dimension.
For example, this can be used to get a new proof of Cobham--Semenov's theorem (see \cite{Durand:2008}) using multidimensional subshifts.
Moreover as subsets of $\mathbb{N}^d$ can be seen as tilings of $\mathbb{R}_+^d$ (which are self-similar, see \cite{Solomyak:1997} for the definition) it is not surprising to have generalizations to self-similar tiling dynamical systems \cite{Cortez&Durand:2008}.

The notion of recognizable subsets of $\mathbb{R}$ or $\mathbb{R}^d$, in standard numeration systems, can easily be defined. 
In a series of papers \cite{Boigelot&Brusten:2007,Boigelot&Brusten&Bruyere:2008,Boigelot&Brusten&Leroux:2009} the authors obtained a very nice generalization of Cobham's theorem.
In dimension one the same result (but in a different setting) has been obtained in \cite{Adamczewski&Bell:xy} independently.

The case of the ring of Gaussian integers with the numeration systems $((-a+i)^ n)_n$, $a\in \mathbb{N}\setminus \{ 0\}$, (see \cite{Katai&Szabo:1975}) has been investigated in  \cite{Hansel&Safer:2003}.
They obtained a very partial result and faced the problem to prove that $\frac{\ln a}{\ln b} , \frac{\tan^{-1} a}{2\pi}  \frac{\ln a}{\ln b} - \frac{\tan^{-1} b}{2\pi}, 1$ are rationally independent which seems to be a difficult number theoretic problem.
The generalization remains open.

In \cite{Allouche&Shallit:1992}, the authors defined the notion of $p$-regular sequences. They take values in a ring and are defined using the notion of $k$-kernel. 
When the ring is finite they are $k$-automatic sequences.
J. Bell generalized Cobham's theorem to this context in \cite{Bell:2005}.

G. Christol gave in \cite{Christol:1979} a famous and very concrete description of the elements of $\mathbb{F}_q ((t ))$ that are algebraic
over $\mathbb{F}_q (t)$ ($q=p^ n$ with $p$ a prime number); it shows that being an algebraic power series is equivalent to the
sequence of coefficients being $p$-automatic.
In \cite{Kedlaya:2006} K. Kedlaya generalized this theorem to the so-called
generalized power series of Hahn ($\mathbb{F}_q ((t^\mathbb{Q}))$) in terms of quasi-automatic functions. 
Then, B. Adamczewski and J. Bell \cite{Adamczewski&Bell:2008} proved an extension of Cobham's theorem to quasi-automatic functions and put it together with Kedlaya's result to derive an analogue of the main result in \cite{Christol&Kamae&MendesFrance&Rauzy:1980} asserting that 
a sequence
of coefficients represents two algebraic power series in distinct characteristics if and only if these
power series are rational functions.

For more details on all these developments we refer to \cite{Adamczewski&Bell:xx,Durand&Rigo:xx}.

\section{Words, morphisms, substitutions and numeration systems}

In this section we recall classical definitions and
notation.

\subsection{Words and sequences} 
An {\it alphabet} $A$ is a finite set of elements called {\it
  letters}. 
A {\it word} over $A$ is an element of the free monoid
generated by $A$, denoted by $A^*$. 
Let $x = x_0x_1 \cdots x_{n-1}$
(with $x_i\in A$, $0\leq i\leq n-1$) be a word, its {\it length} is
$n$ and is denoted by $|x|$. 
The {\it empty word} is denoted by $\epsilon$, $|\epsilon| = 0$. 
The set of non-empty words over $A$ is denoted by $A^+$. 
The elements of $A^{\NN}$ are called {\it sequences}. 
If $x=x_0x_1\cdots$ is a sequence (with $x_i\in A$, $i\in \NN$) and $I=[k,l]$ an interval of
$\NN$ we set $x_I = x_k x_{k+1}\cdots x_{l}$ and we say that $x_{I}$
is a {\it factor} of $x$.  If $k = 0$, we say that $x_{I}$ is a {\it
 prefix} of $x$. 
The set of factors of length $n$ of $x$ is written
$\mathcal{L}_n(x)$ and the set of factors of $x$, or the {\it language} of $x$,
is noted $\mathcal{L}(x)$. 
The {\it occurrences} in $x$ of a word $u$ are the
integers $i$ such that $x_{[i,i + |u| - 1]}= u$. 
If $u$ has an occurrence in $x$, we also say that $u$ {\em appears} in $x$.
When $x$ is a word,
we use the same terminology with similar definitions.

The sequence $x$ is {\it ultimately periodic} if there exist a word
$u$ and a non-empty word $v$ such that $x=uv^{\omega}$, where
$v^{\omega}= vvv\cdots $. 
It is {\it periodic} if $u$ is the empty word. 
A word $u$ is {\em recurrent} in $x$ if it appears in $x$ infinitely many times.
The set of recurrent words of $x$ is denoted by $\mathcal{L}_{\rm rec} (x)$.
A sequence $x$ is {\it uniformly recurrent} if every factor $u$ of $x$
appears infinitely often in $x$ and the greatest
difference of two successive occurrences of $u$ is bounded.

\subsection{Morphisms and matrices} 
Let $A$ and $B$ be two alphabets. Let $\tau$ be a {\it morphism} 
from $A^*$ to $B^*$. Such a map induces by concatenation a morphism from
$A^*$ to $B^*$.  
When $\tau (A) = B$, we say $\tau$ is a {\em coding}. 
Thus, codings are onto.
If $\tau (A)$ is included in $B^+$, it induces by concatenation a map
from $A^{\NN}$ to $B^{\NN}$. These two maps are also called $\tau$.
With the morphism $\tau$ is naturally associated the matrix $M_{\tau} =
(m_{i,j})_{i\in B , j \in A }$ where $m_{i,j}$ is the number of
occurrences of $i$ in the word $\tau(j)$.

It is well-known that any non-negative square matrix $M$ has a real eigenvalue $\alpha$ which
is real and greater or equal to the modulus of any other eigenvalue.  
We call $\alpha$ the {\em dominating eigenvalue} of $M$. 
Moreover $\alpha$ is a {\em Perron number}: it is an algebraic real number $>1$
strictly dominating the modulus of all its algebraic conjugates (see
for instance \cite{Lind&Marcus:1995}).
The matrix $M$ is called {\it primitive} if it has a power such that all its coefficients are positive. 
In this case the dominating eigenvalue is unique, positive and it is a simple
root of the characteristic polynomial. 
This is Perron's Theorem.

\subsection{Substitutions and substitutive sequences}

A {\it substitution} is a morphism $\tau : A^* \rightarrow
A^{*}$. 
If there exist a letter $a\in A$ and a word $u\in A^+$ such that $\sigma(a)=au$ and moreover, if $\lim_{n\to+\infty}|\sigma^n(a)|=+\infty$, then $\sigma$ is said
    to be {\em prolongable on $a$}.

Since for all $n\in\mathbb{N}$, $\sigma^n(a)$ is a prefix of $\sigma^{n+1}(a)$ and because
     $|\sigma^n(a)|$ tends to infinity with $n$, the
     sequence $(\sigma^n(aaa\cdots ))_{n\ge 0}$ converges (for the usual
     product topology on $A^\mathbb{N}$) to a sequence denoted by
     $\sigma^\omega(a)$. 
The morphism $\sigma$ being continuous for the product topology, $\sigma^\omega(a)$ is a fixed point of $\sigma$: $\sigma (\sigma^\omega(a)) = \sigma^\omega(a)$.
A sequence obtained in
    this way by iterating a prolongable substitution is said to be {\em purely substitutive} (w.r.t. $\sigma$). 
If $x\in
    A^\mathbb{N}$ is purely substitutive and if $\phi:A^*\to B^*$ is a
    coding, then the sequence $y=\phi (x)$ is said to be {\em
      substitutive}. 

Whenever the matrix associated with $\tau $ is
primitive we say that $\tau$ is a {\it primitive substitution}.  We
say $\tau$ is a {\it growing} substitution if $\lim_{n\rightarrow
  +\infty} |\tau^n (b)| = +\infty$ for all $b\in A$. It is
{\it erasing} if there exists $b\in A$ such that $\tau (b)$ is the
empty word.  

\begin{defi}
\label{defi:alphasub}
Let $A$ be a finite alphabet.
A sequence $x\in B^\mathbb{N}$ is said to be $\alpha$-{\em substitutive} (w.r.t. $\sigma$) if $\sigma : A^*\to A^*$ is a substitution prolongable on the letter $a$ such that:

\begin{enumerate}
\item
\label{cond1sub}
all letters of $A$ have an occurrence in $\sigma^\omega (a)$;  
\item
\label{cond2sub}
$\alpha$ is the dominating eigenvalue of the incidence matrix of $\sigma$;
\item
there exists a coding $\phi : A^*\to B^*$ with  $x=\phi (\sigma^\omega (a))$.
\end{enumerate}
  
If moreover $\sigma$ is primitive, then $\phi
    (\sigma^\omega (a))$ is said to be a {\em primitive
      $\alpha$-substitutive infinite sequence (w.r.t. $\sigma$)}.
\end{defi}

The condition \eqref{cond1sub} is important.
Indeed, let consider the substitution $\tau $ defined by $\tau (a)=aaab$,
$\tau (b) = bc$ and $\tau (c) = cb$.
It has three fixed points $\tau^\omega (a)$, $\tau^\omega (b)$ and $\tau^\omega (c)$.
The sequence $\tau^\omega (a)$ is $3$-substitutive and 
we do not want to say that $\tau^\omega (b)$ and $\tau^\omega (c)$ are $3$-substitutive.
With our definition they are $2$-substitutive.

\subsection{Growth type and erasures}\label{subsec:growth}

The following well known result (see Chapter III.7 in \cite{Salomaa&Soittola:1978}) will be very useful in the sequel.

\begin{prop}
\label{prop:croissance}
Let $\sigma : A^*\to A^*$ be a substitution.
For all $a\in A$, one of the following two situations occurs

\begin{enumerate}
\item
$\exists N\in{\mathbb N} : \forall n>N,\ |\sigma^n(a)|=0$, or,
\item
\label{Ggrowthorder}
there
exist $d(a)\in{\mathbb N}$ and real numbers $c(a),\theta(a)$ such that

$$
\lim_{n\to+\infty} \frac{|\sigma^n(a)|}{c(a)\, n^{d(a)}\, \theta(a)^n}=1.
$$
\end{enumerate}    

Moreover, under the situation \eqref{Ggrowthorder},
    for all $i\in\{0,\ldots,d(a)\}$ there exists a letter $b\in A$
    appearing in $\sigma^j(a)$ for some $j\in{\mathbb N}$ satisfying

$$
\lim_{n\to+\infty} \frac{|\sigma^n(b)|}{c(b)\, n^i\,
      \theta(a)^n}=1.
$$
\end{prop}

This justifies the following definition.

\begin{defi} 
Let $\sigma : A^* \to A^*$ be a non-erasing substitution. 
For all $a\in A$, the couple $(d(a), \theta(a))$ defined in Proposition \ref{prop:croissance} is called the {\em growth type} of $a$ and $\theta (a)$ is called the {\em growth rate} of $a$ (w.r.t. $\sigma$). 
The growth type of a word is the maximal growth type of its letters.
If
$(d,\theta)$ and $(e,\beta)$ are two growth types we say that
$(d,\theta)$ is {\em less than} $(e,\beta)$ (or $(d,\theta) <
(e,\beta)$) whenever $\theta < \beta$ or, $\theta = \beta$ and $d<e$.
\end{defi}

We say that $a\in A$ is a {\em growing letter}\index{growing letter}
if $(d(a),\theta (a))>(0,1)$ or equivalently, if $\lim_{n\to +\infty}
|\sigma^n (a)| = +\infty $.

We set $\Theta := \max \{ \theta(a) \mid a\in A \}$, $D := \max \{ d(a) \mid 
a\in A, \theta(a) = \Theta  \}$ and $A_{max} := \{a\in A \mid
\theta (a) = \Theta , d(a) = D \}$.  
The dominating eigenvalue of $M$ is
$\Theta$.  
Consequently, any sequence which is substitutive w.r.t. $\sigma$ is $\Theta$-substitutive.
We will say that the letters of $A_{max}$ are {\em of
  maximal growth}\index{maximal growth} and that $(D,\Theta)$ is the {\em growth type}\index{growth type} of
$\sigma$. 
Observe that if $\Theta=1$, then in view of the last part of Proposition
\ref{prop:croissance}, there exists at least one non-growing letter (of
growth type $(0,1)$).  
Otherwise stated, if a letter has a polynomial
growth, then there exists at least one non-growing letter.
Consequently $\sigma$ is growing ({\em i.e.}, all its letters are growing)
if and only if $\theta (a) > 1$ for all $a\in A$.  
We observe that for all $k\geq 1$, the growth type of $\sigma^k $ is $(D, \Theta^k )$.

The following theorem allows us to suppose that the substitutions we deal with are non-erasing.
We refer to \cite[Th\'eor\`eme 4]{Cassaigne&Nicolas:2003} for the proof of this theorem even if Property (3) is not stated in this paper but is clear from the proof. 
Other references with other proofs are in \cite{Cobham:1968,Pansiot:1982,Allouche&Shallit:2003,Honkala:2009}.

\begin{theo}
\label{theo:Cassaigne&Nicolas}
Let $\sigma : A^*\to A^*$ be a substitution prolongable on the letter $a$ and $\phi : A^* \to C^*$ be a morphism such that $\phi (\sigma^\omega (a))$ belongs to $A^\mathbb{N}$. 
Then, $\phi (\sigma^\omega (a)) = \psi (\tau^\omega (b))$ where:

\begin{enumerate}
\item
$\psi : B^* \to C^*$ is a coding;
\item
$\tau : B^* \to B^*$ is a non-erasing substitution prolongable on $b$;
\item
there exist $\gamma : A\to B^*$ and $k$ such that 

$$\gamma \circ \sigma^k = \tau \circ \gamma \ {\rm and } \ 
\psi \circ \gamma = \phi .
$$
\end{enumerate} 
\end{theo}

From the classical theory of non-negative matrices (see \cite{Horn&Johnson:1990}) we deduce the following corollary.

\begin{coro}
\label{coro:nonerasing}
The image by a non-erasing morphism of an $\alpha$-substitutive sequence is an $\alpha^k$-substitutive sequence with respect to a non-erasing substitution for some $k$. 
\end{coro}

\begin{proof}
Let $\sigma : A^*\to A^*$ be a substitution prolongable on the letter $a$ and $\alpha$ the dominating eigenvalue of its incidence matrix.
Let $y = \sigma^\omega (a)$ and $\phi : A^* \to C^*$ be a non-erasing morphism.
We can suppose all letters of $C$ appear in some words of $\phi (A)$.
It suffices to show that $\phi (y)$ is $\alpha^k$-substitutive for some $k$.

From Theorem \ref{theo:Cassaigne&Nicolas} we can suppose $\phi (y)= \psi (\tau^\omega (b))$ where $\psi$ and $\tau$ satisfy (1), (2) and (3) of this theorem.
It suffices to show that the dominating eigenvalue of the incidence matrix $M_\tau $ of $\tau$ is $\alpha^k$ (where $k$ comes from (3) of Theorem \ref{theo:Cassaigne&Nicolas}).
Let $M_\phi$, $M_\psi$ and $M_\sigma$ be the incidence matrices of the corresponding morphisms.
Let $v$ be an eigenvector of $M_\sigma$ for the eigenvalue $\alpha$.
Then, $\alpha^k M_\gamma v = M_\tau M_\gamma v$.
As $\phi$ is non-erasing and $M_\psi M_\gamma = M_\phi $, $M_\gamma v$ is a non-zero vector.
Hence it is an eigenvector of $M_\tau$ for the eigenvalue $\alpha^k$.

Conversely, let $\beta$ be the dominating eigenvalue of $M_\tau$ and $w$ be a corresponding left eigenvector.
Then, $w M_\gamma M_\sigma^k = \beta w M_\gamma $.
As $\psi$ is a coding, there exists $w'$ with non-negative coordinates such that $w=w'M_\psi$.
Thus, $w' M_\phi M_\sigma^k = \beta w' M_\phi $.
All letters of $C$ appearing in some words of $\phi (A)$, $w' M_\phi$ is a non-negative vector different from $0$.
This concludes the proof. 
\end{proof}

Once Theorem \ref{theo:main} will be proven, it will show that, in this corollary, the non-erasing assumption cannot be removed. 
For example, we can consider $\sigma$ defined by $\sigma (a)= ab$, $\sigma (b)=bac$ and $\sigma (c) = ccc$, and, the (erasing) morphism $\phi$ defined by $\phi (a) = a$, $\phi (b) = b$ and $\phi (c)$ is the empty word. 
Then, $\sigma^\omega (a)$ is $3$-substitutive while $\phi (\sigma^\omega (a))$ is a non ultimately periodic $2$-substitutive sequence.

\section{Proof of Theorem \ref{theo:main}}
\label{sec:maintheo}
\subsection{Already known results used to prove the conjecture}

The proofs of most of the generalizations of Cobham's theorem are
divided into two parts.
\begin{itemize}
  \item[\textup{(i)}] Dealing with a subset $X$ of integers, we have
    to prove that $X$ is syndetic. Equivalently, dealing with an
    infinite sequence $x$, we have to prove that the letters occurring
    infinitely many times in $x$ appear with bounded gaps.
  \item[\textup{(ii)}] In the second part of the proof, the ultimate
    periodicity of $X$ or $x$ has to be carried on.
\end{itemize}

The first part is already known in the framework of substitutions.

\begin{theo}
\label{the:boundedgap}
\cite[Theorem 17]{Durand&Rigo:2009}
Let $\alpha , \beta \in ] 1,+\infty [$ be two multiplicatively independent real numbers. 
If a sequence $x$ is both $\alpha$-substitutive and $\beta$-substitutive then the words having
infinitely many occurrences in $x$ appear in $x$ with bounded gaps.
\end{theo}

The second part is proved in the context of ``good'' substitutions in \cite{Durand:2002a} and was proven previously in many others (see \cite{Durand&Rigo:xx} for a survey). 

\begin{defi}
    Let $\sigma : A^*\rightarrow A^*$ be a substitution whose
    dominating eigenvalue is $\alpha$. If there exists $B\subseteq A$
    such that for all $b\in B$, $\sigma(b)\in B^*$, then the
    substitution $\tau:B^*\to B^*$ defined by $\tau(b)=\sigma(b)$, for
    all $b\in B$, is a {\em
      sub-substitution}
    of $\sigma$. The substitution $\sigma$ is a {\em ``good''
      substitution} if it is growing and has a primitive sub-substitution whose
    dominating eigenvalue is $\alpha$.  
\end{defi}

Not all non-erasing substitutions have a primitive sub-substitution, for example $0\mapsto 010$, $1\mapsto 2$ and $2\mapsto 1$. But all have a power that has at least a primitive sub-substitution (\cite{Durand:2002a}).
But, even up to some power, some are not ``good''.
For example the substitution $\sigma$ defined by $0\mapsto 0100$, $1\mapsto 12$ and $2\mapsto 21$ has a unique primitive sub-substitution ($\tau : 1\mapsto 12$ and $2\mapsto 21$) but is not ``good'' because the dominant eigenvalue of $\sigma$ is $3$ and it is $2$ for $\tau$.

We recall that in the (non-erasing) purely substitutive context the expected extension of Cobham's theorem is known.

\begin{theo}
\cite[Corollary 19]{Durand:2002a}
\label{theo:fixpoint}
Let $\sigma : A^* \rightarrow A^*$ and $\tau : A^* \rightarrow A^*$ be
two non-erasing growing substitutions prolongable on $a\in A$ with
respective dominating eigenvalues $\alpha$ and $\beta$.  
Suppose that
all letters of $A$ appear in $\sigma^{\omega} (a)$ and in $\tau^\omega
(a)$ and that $\alpha$ and $\beta$ are multiplicatively independent.
If $x=\sigma^{\omega} (a) = \tau^\omega (a)$, then $x$ is ultimately
periodic.
\end{theo}

\subsection{Concatenation of return words}

Let $A$ be a finite alphabet.
Let $x\in A^\mathbb{N}$ and $u\in \mathcal{L}(x)$. 
A {\it return word to $u$ (for $x$)} is a word $w$ such that $wu$ belongs to $\mathcal{L}(x)$, $u$ is a prefix of $wu$ and $u$ has exactly two occurrences in $wu$. 
The set of return words is denoted by $\mathcal{R}_x (u)$.
A sequence $x\in A^\mathbb{N}$ is {\it linearly recurrent (with constant $L$)} if $x$ is uniformly recurrent and for all $w\in \mathcal{R}_x (u)$ we have 
$|w| \leq  L |u|$.

\begin{lemma}
\label{lemma:lr}
\cite[Theorem 24]{Durand&Host&Skau:1999}
Let $x\in A^\mathbb{N}$ be a non-periodic linearly recurrent sequence (with constant $L$).
Then, 
\begin{enumerate}
\item
for all $w\in \mathcal{R}_x (u)$ we have 
$ \frac{|u|}{L} \leq |w|$;
\item
there exists a constant $K$ such that for all $u\in \mathcal{L}(x)$ we have $\# \mathcal{R}_x (u)\leq K$. 
\end{enumerate}
\end{lemma}

The following lemma will be, in some sense, the last decisive argument to prove Theorem \ref{theo:main}.
When $W$ is a set of words, $W^*$ stands for the set of all concatenations of elements of $W$.

\begin{lemma}
\label{lemma:concatenationrw}
Let $x$ be a non-periodic linearly recurrent sequence (with constant $L$).
Then, there exists a constant $K$ such that for all $u\in \mathcal{L}(x)$ and all $l\in \mathbb{N}$, 

$$
\# \cup_{0\leq n \leq l} 
\left( 
\mathcal{R}_x(u)^* \cap \mathcal{L}_n(x)
\right)
\leq \left( 1+ K \right)^{\frac{lL}{|u|}} .
$$
\end{lemma}

\begin{proof}
Let $K$ be the constant given by Lemma \ref{lemma:lr}.
Let $l,n\in \mathbb{N}$ with $n\leq l$, $u\in \mathcal{L} (x)$ and $w\in \mathcal{R}_x(u)^* \cap \mathcal{L}_n(x)$.
As the distance between two occurrences of $u$ in $w$ is at least $|u|/L$,
in $w$ there are at most $Ll/|u|$ occurrences of $u$.
Hence $w$ is a concatenation of exactly $Ll/|u|$ words belonging to $\mathcal{R}_x (u) \cup \{\epsilon \}$.
This concludes the proof.
\end{proof}

The following proposition allows us to apply this lemma to primitive substitutive sequences.

\begin{theo}
\label{prop:subprimlr}
\cite[Theorem 4.5]{Durand:1998a}
All primitive substitutive sequences are linearly recurrent.
\end{theo}

Note that it is explained in \cite{Durand:1996} through the Chacon example that uniformly recurrent substitutive sequences are primitive substitutive. In \cite{Durand:2010} it is explained the generality of the treatment of the Chacon substitution. 
The same result has been obtained in \cite{Damanik&Lenz:2006} but with a significant longer proof.

\subsection{Reduction of the problem}
\label{subsec:reduction}

Let us make two obvious but important remarks.
Let $\alpha , \beta >1$.
When a sequence is $\alpha$-substitutive with respect to a substitution $\sigma$, it is also $\alpha^k$-substitutive with respect to $\sigma^k$. 
Moreover, for all positive integers $l$ and $k$, $\alpha$ and $\beta$ are multiplicatively independent if and only if $\alpha^l$ and $\beta^k$ are multiplicatively independent.
Hence due to the statement we want to prove, we can suppose, if needed, that $\sigma$ has the properties some $\sigma^k$ would have (changing $\sigma$ by $\sigma^k$ if needed).
Hence we can always suppose that $\sigma $ fulfills:

\begin{enumerate}
\item
\label{enum:nonerasing}
$\sigma $ is non-erasing (Corollary \ref{coro:nonerasing});
\item
\label{enum:pansiot}
$\sigma$ is growing or there exists a growing letter $a\in A$ such that $\sigma (a) =
  vau$ (or $uav$) with $u\in B^*\setminus \{ \epsilon \}$ where $B$ is the set of non-growing letters (\cite[Th\'eor\`eme 4.1]{Pansiot:1984a});
\item
\label{enum:primitivesubsub} 
$\sigma$ has a primitive sub-substitution (see \cite{Durand:2002a}).
\end{enumerate}

\subsection{Technical lemmata and proof of Theorem \ref{theo:main}}

Before to prove the main result of this paper we need to establish some lemmata.

\begin{lemma}
\label{lemma:langreclr}
Let $x$ be a substitutive sequence with respect to a growing substitution.
If each word in $\mathcal{L}_{\rm rec} (x)$ appears with bounded gaps in $x$ then there exists a primitive substitutive sequence $y$ such that 
$\mathcal{L}_{\rm rec} (x)=\mathcal{L} (y)$.
\end{lemma}

\begin{proof}
Let $\sigma : A^* \to A^*$ be a growing substitution prolongable on the letter $a$ and $\phi : A^* \to B^*$ be a coding such that $x=\phi (\sigma^\omega (a))$.
We can suppose that all letters of $A$ occurs in some $\sigma^n (a)$, $n\in \mathbb{N}$.
Proposition 15 in \cite{Durand:2002b} asserts that growing substitutions have at least one primitive sub-substitution.
Let $\tau : C^*\to C^*$ be a primitive sub-substitution of $\sigma$.
It is necessarily a growing substitution and it is easy to check that each word $\sigma^n (b)$, $b\in C$, $n\in \mathbb{N}$, is recurrent in  $\sigma^\omega (a)$. 
There exist $c\in C$ and a positive integer $l$ such that $c$ is a prefix of $\sigma^l (c)$.
Thus $y = \phi \left(\left( \sigma^l\right)^\omega (c)\right)$ exists, is primitive substitutive and $\mathcal{L} (y) \subset \mathcal{L}_{\rm rec} (x)$.
As primitive substitutive sequences are uniformly recurrent, from the hypothesis we deduce that $\mathcal{L}_{\rm rec} (x)=\mathcal{L} ( y) $.
\end{proof}

\begin{lemma}
\label{lemma:classicalarguments}
Let $x$ be a sequence and $u$ be a word such that: if $u=v^k$ then $k=1$.
Suppose that each word in $\mathcal{L}_{\rm rec} (x)$ appears in $x$ with bounded gaps and that $\mathcal{L}_{\rm rec} (x)$ contains $\{ u^n ; n\in \mathbb{N} \}$. 
Then, $x$ is ultimately periodic and $\mathcal{L}_{\rm rec} (x) = \mathcal{L} ( uuu\cdots )$.
\end{lemma}

\begin{proof}
Let $u$ be such that: if $u=v^k$ then $k=1$.
Let us reproduce arguments already used in \cite[Theorem 18]{Durand:2002a} in order to conclude that $x$ is ultimately periodic.
As $u$ belongs to $\mathcal{L}_{\rm rec} (x)$, it appears with bounded gaps in $x$.
Thus, the set $\mathcal{R}_{x}(u)$ of return words to $u$ is finite. 
There exists an integer $N$ such that all the words $wu\in \mathcal{L}(x_N x_{N+1} \cdots )$, $w\in \mathcal{R}_{x} (u)$, appear infinitely many times in $x$. 
Hence these words appear with bounded gaps in $x$. 
We set $t = x_N x_{N+1} \cdots $. 
We can suppose that $u$ is a prefix of $t$. 
Then $t$ is a concatenation of return words to $u$. 
Let $w$ be a return word to $u$ such that $wu$ belongs to $\mathcal{L} (t)$.
It appears in some power of $u$: $wu = su^kp$ where $k\geq 1$, $s$ is a suffix of $u$, with $|s|<|u|$, and $p$ a prefix of $u$, with $|p|<|u|$.
Hence $p$ is also a suffix of $u$ and there exists $p'$ such that $u=pp'=p'p$.
If $p$ and $p'$ are non-empty then $u=v^k$ for some $v$ and $k\geq 2$ (see \cite[Proposition 1.3.2]{Lothaire:1983}).
This is in contradiction with our assumption. 
Consequently, $p$ or $p'$ is empty. 
Doing the same with $s$ we will deduce that necessarily $w=u$.
It follows that $t = u^{\omega}$, $x$ is ultimately periodic and $\mathcal{L}_{\rm rec} = \mathcal{L} (t) = \mathcal{L} (uuu\cdots )$.
\end{proof}

We have seen that Theorem \ref{theo:main} was proven in \cite{Durand:2002a} in the context of ``good'' substitutions. 
The following lemma, together with Theorem \ref{the:boundedgap}, is the main argument to treat, in Theorem \ref{theo:main}, the case where one of the two substitutive sequences is not ``good''.

\begin{lemma}
\label{lemma:finallemma}
Let $x$ be purely substitutive w.r.t. $ \sigma : A^* \to A^*$ satisfying \eqref{enum:nonerasing} and \eqref{enum:pansiot} of Section \ref{subsec:reduction}, having a dominating eigenvalue $\alpha$ strictly greater than $1$ and such that all letters of $A$ occurs in  $x$. 
Suppose there exists a coding $\phi : A^*\to B^*$ such that all words belonging to $\mathcal{L}_{\rm rec} (\phi (x))$ appear with bounded gaps in $\phi (x)$.
Then, either all letters of $A$ have the same growth rate (w.r.t $\sigma$) or $\phi (x)$ is ultimately periodic.
\end{lemma}

\begin{proof}
Suppose $\sigma$ is prolongable on $a'$, $x=\sigma^\omega (a')$ and $y=\phi (x)$.
Suppose $\sigma $ has at least two growth rates and let us show $\phi (x) $ is ultimately periodic.
The growth type of $\sigma $ is $(d,\alpha )$ for some $d$.

We consider two cases.

Suppose $\sigma $ is not a growing substitution. 
Then, from Assumption \eqref{enum:pansiot} in Section~\ref{subsec:reduction}, there exists a growing letter $c\in A$ such that $\sigma (c) =
  vcu$ (or $ucv$) with $u\in B^*\setminus \{ \epsilon \}$ where $B$ is the set of non-growing letters.
It is convenient to notice that $\sigma^n (c) = \sigma^{n-1} (v)cu\sigma (u) \cdots \sigma^{n-1} (u)$ or $\sigma^n (c) = \sigma^{n-1} (u) \cdots \sigma (u) u c\sigma^{n-1} (v)$ and there exist two distinct positive integers $i$ and $j$ such that $\sigma^i (u)=\sigma^j (u)$.
 
  Let us show that $c$ necessarily belongs to $\mathcal{L}_{\rm rec} (x)$ whenever $x$ is not ultimately periodic.
 
 Suppose $c=a'$ and  $\sigma (c) = vcu$. 
 If $v$ is not the empty word then $\sigma (c) = cv'cu$ and $c$ is clearly recurrent in $x$.
 If $v$ is the empty word then $\sigma^n (c) = cu \sigma (u) \cdots \sigma^{n-1} (u)$.
  Consequently, because $\sigma^i (u)=\sigma^j (u)$,  $x$ is ultimately periodic.
 Suppose $c=a'$ and  $\sigma (c) = ucv$.
 Then $a'$ is a non-growing letter. This contradicts our assumptions.
 Suppose $c\not = a'$. Then $\sigma^m (a') = a'sct$, for some $m$, and $c$ is clearly recurrent.
  Finally we can consider $c$ is recurrent in $x$.
  
The letter $c$, and consequently $\sigma^n ( c)$, $n\in \mathbb{N}$, having infinitely many occurrences in $x$, it is also the case for the words $u^{(n)}=(\sigma^i (u) \sigma^{i+1} (u) \cdots \sigma^{j-1} (u))^n$, $n\in \mathbb{N}$.
Then all words in $\left\{ \left(\phi (u^{(1)})\right)^n  ; n\in \mathbb{N} \right\}$ appear infinitely many times in $\phi (x)$.
Lemma \ref{lemma:classicalarguments} implies that $\phi (x)$ is ultimately periodic.

\medskip

Suppose now that $\sigma $ is a growing substitution.
Then, all letters of $A$ have an exponential growth rate $\gamma$ with $\gamma >1$.
From the assumptions,
the numbers

$$\beta  = \max \{  \theta (a)  \mid \theta (a) < \alpha , a \in A \} \ \  {\rm  and } \ \  
f = \max \{  d(a) \mid \theta (a) =\beta , a \in A \}
$$ 

exist and $(f,\beta)$ is a growth type of some letter of $A$ with $1<\beta < \alpha$.

It necessarily exist two letters occurring infinitely maning times in $x$, one having growth type $(f,\beta )$ and the other having an exponential growth rate $\alpha$.  
Hence, there exists a non-empty word $awb$, appearing infinitely many times in $x$, where $w$ has growth type $(f , \beta )$, and, $a$ and $b$ have an exponential growth rate $\alpha$. 
Once we observe that for all letters $c$ with exponential growth rate $\alpha$ there exist letters $c',c''$ with exponential growth rate $\alpha$ such that $\sigma (c) = u'c'u=vc''v'$ where $u$ and $v$ (being possibly the empty word) have exponential growth rates strictly less than $\alpha$, then,
by a recurrence starting with $awb$, it is easy to prove that there exist two sequences of letters $(a_n)_n$ and $(b_n)_n$ and two sequences of words $(u_n)_n$ and $(v_n)_n$ such that for all $n\in \mathbb{N}$:

\begin{enumerate}
\item
$a_n u_n \sigma (u_{n-1}) \cdots \sigma^{n-1} (u_1) \sigma^{n} (w) \sigma^{n-1} (v_1) \cdots \sigma (v_{n-1}) v_n b_n $
appears infinitely many times in $x$;
\item
$a_n$ has growth type $(d_n , \alpha)$;
\item
$b_n$ has growth type $(e_n , \alpha)$;
\item
$u_n$ has growth type $(f_n,\beta_n)$ with $\beta_n<\alpha$ and $|u_n|< \max_{c\in A} |\sigma (c)|$;
\item
$v_n$ has growth type $(g_n,\gamma_n)$ with $\gamma_n<\alpha$ and $|v_n|< \max_{c\in A} |\sigma (c)|$.
\end{enumerate}

Consequently, for all $n$ and $k$ the word $W(n,k) = \sigma^k (a_n)U(n,k)\sigma^k (b_n)$, where

$$
\begin{array}{l}
U(n,k) = \sigma^k(u_n) \cdots \sigma^{k+n-1} (u_1) \sigma^{k+n} (w) \sigma^{k+n-1} (v_1) \cdots  \sigma^k (v_n)  ,
\end{array}
$$

appears infinitely many times in $x$. 
There exist a strictly increasing sequence $(m_i)$ and two letters $a$ and $b$ such that 

\begin{equation}
\label{proof:aprime}
a_{m_i} = a'' \ \ {\rm and} \ \  b_{m_i} = b'' \ \ \hbox{\rm for all} \ \ i.
\end{equation}

From Lemma \ref{lemma:langreclr},  
there exists a linearly recurrent sequence $z$ such that 

\begin{enumerate}
\item[(P1)]
$\mathcal{L}(z) = \mathcal{L}_{\rm rec} (y)$.
\end{enumerate}

If $z$ is periodic, we achieve the proof with Lemma \ref{lemma:classicalarguments}. 
Hence we suppose $z$ is not periodic.

From the lemmata \ref{lemma:lr} and \ref{lemma:concatenationrw} there exist some constants $K$, $d$ $e$, such that:
\begin{enumerate}
\item[(P2)]
for all $u\in \mathcal{L}_{\rm rec} (y)$  and $w\in \mathcal{R}_y (u)\cap \mathcal{L}_{\rm rec} (y)$ we have 
$ \frac{|u|}{K} \leq |w|\leq K|u|$;
\item [(P3)]
for all $u\in \mathcal{L}_{\rm rec}(y)$ we have $\# \mathcal{R}_y (u)\cap \mathcal{L}_{\rm rec} (y)\leq K$;
\item [(P4)]
for all $u\in \mathcal{L}_{\rm rec}(y)$ and all $l\in \mathbb{N}$;
$$
\# \cup_{0\leq n \leq l} \left( \mathcal{R}_y(u)^* \cap \mathcal{L}_n(y)\cap \mathcal{L}_{\rm rec} (y)\right) \leq \left( 1+ K \right)^{\frac{lK}{|u|}} .
$$
\item [(P5)]
for all $n$ and $k$
\begin{align*}
\label{proof:aprimeprime}
|U(n,k)|   & \leq K\left( (n+k)^{f+1} \beta^{n+k} \right) ,\\
\frac{1}{K} k^{d} \alpha^{k}  & \leq |\sigma^k (a'')| \leq K k^{d}  \alpha^{k} \hbox{ and }\\
\frac{1}{K} k^{e}\alpha^{k}   & \leq |\sigma^k (b'')| \leq K k^{e}  \alpha^{k} .
\end{align*}
\end{enumerate}

From \eqref{proof:aprime} and the previous inequalities, there exists $k$ such that $|\sigma^k (b'')|\leq |\sigma^k (a'')|$ (the other case can be treated in the same way), $2(1+K)\leq |\sigma^k (b'')|$ and 

\begin{align}
|U(m_i,k)| \leq |\sigma^k (b'')| \hbox{ for all } 1\leq i\leq \left( 1+ K \right)^{2(K+3)(K+1)K} +1 .
\end{align}

From (P1) and (P2), for all $j$, all words in $\mathcal{L}_{\rm rec} (y)\cap \mathcal{L}_j (y)$ appear in all words of $\mathcal{L}_{\rm rec} (y)\cap \mathcal{L} _{(K+1)j} (y)$.
Let $u$ be a prefix of $\phi (\sigma^k (b''))$ such that  

$$
\frac{|\sigma^k (b'')|}{K+1} -1 \leq |u| \leq  \frac{|\sigma^k (a'')|}{K+1} .
$$

Then $u$ is non-empty and occurs in $\phi (\sigma^k (a''))$.
We can decompose $\phi (\sigma^k (a''))$ and $\phi (\sigma^k (b''))$ in such a way that $\phi (\sigma^k (a'')) = Q_puQ_s$ and $\phi (\sigma^k (b'')) = uR_s$ with

$$
|Q_s|\leq (K+1)|u| .
$$

Observe that $u$, $a''$, $b''$ and $Q_s$ do not depend on $i$.
Moreover, for all $i$ belonging to $[1 , \left( 1+ K \right)^{2(K+3)(K+1)K} +1]$, the word $uQ_s \phi (U(m_i,k))$ belongs to $\mathcal{L}_{\rm rec} (y)$, is a concatenation of return words to $u$ and satisfies

$$
|uQ_s \phi (U(m_i,k))| \leq (K+3)|\sigma^k (b'')|  
. $$

Moreover,

\begin{align}
\label{ineg:rwcontradict}
\# \bigcup_{n=0}^{(K+3)|\sigma^k (b'')|} \mathcal{R}_u(y)^* \cap \mathcal{L}_n(y) 
& \leq 
\left( 1+ K \right)^{\frac{(K+3)|\sigma^k (b'')|K}{|u|}}
 \leq 
\left( 1+ K \right)^{2(K+3)(K+1)K} .
\end{align}

But observe that $(|U(m_i,k)|)_{0\leq i \leq 1 + ( 1+ K )^{2(K+3)(K+1)K}}$ being strictly increasing, the words 
$uQ_s \phi (U(m_i,k)) $, $0\leq i \leq 1 + ( 1+ K )^{2(K+3)(K+1)K}$ are all distinct (and belong to $\cup_{0\leq n \leq (K+3)|\sigma^k (b'')|  } \mathcal{R}_u(y)^* \cap \mathcal{L}_n(y)$). 
This is in contradiction with \eqref{ineg:rwcontradict}.
\end{proof}

\begin{proof}[Proof of Theorem \ref{theo:main}]
The sufficient condition is proven in \cite[Proposition 7]{Durand:2002a}.
The necessary condition is proven in \cite[Theorem 18]{Durand:2002a} in the context of ``good'' substitution. 
We conclude reducing the problem to better assumptions (Subsection \ref{subsec:reduction}), then using Theorem \ref{the:boundedgap} and Lemma \ref{lemma:finallemma} (after we observe that substitutions having a primitive substitution and that are not "good" necessarily have at least two growth rates).
\end{proof}

\subsection{For morphisms instead of codings}

We start with an example. 
Let $\sigma : \{0,1 \}^* \to \{0,1 \}^*$ be defined by $\sigma (0)=01$ and $\sigma (1)=0$.
The sequence $x=\sigma^\omega (0)$ is $\frac{1+\sqrt{5}}{2}$-substitutive.
Now let $\tau : \{a,0,1 \}^* \to \{a, 0,1 \}^*$ be defined by $\tau (a) = a0a$, $\sigma (0)=a01$ and $\sigma (1)=a0a$.
The sequence $y= \sigma^\omega (a)$ is $3$-substitutive.
But $x = \phi (y)$ where $\phi : \{a,0,1 \}^* \to \{0,1 \}^*$ is defined by $\phi (0) = 0$, $\phi (1) = 1$ and $\phi (a)$ is the empty word. 
Hence a Cobham like theorem does not hold for erasing coding morphisms instead of coding.
But thanks to Corollary \ref{coro:nonerasing} (and Theorem \ref{theo:Cassaigne&Nicolas})  it holds for non-erasing morphisms. 

\begin{theo}
Let $\alpha >1$ and $\beta >1$ be two multiplicatively independent Perron numbers, and, $\phi : A^* \to C^*$ and $\psi : B^* \to C^*$, defined on finite alphabets, be non-erasing morphisms.
Suppose $x\in A^\mathbb{N}$ is $\alpha$-substitutive and $y\in B^\mathbb{N}$ is $\beta$-substitutive.
If $\phi (x) = \psi (y)$ then $\phi (x)$ is ultimately periodic.
\end{theo}

\bigskip

{\bf Acknowledgements.}
The author would like to thank M. Rigo and V. Berth\'e for their valuable reading of the first draft.

\bibliographystyle{alpha}
\bibliography{totalcobham}

\begin{thebibliography}{CKMR80}

\bibitem[AB08]{Adamczewski&Bell:2008}
B.~Adamczewski and J.~Bell.
\newblock Function fields in positive characteristic: expansions and {C}obham's
  theorem.
\newblock {\em J. Algebra}, 319:2337--2350, 2008.

\bibitem[AB10a]{Adamczewski&Bell:xy}
B.~Adamczewski and J.~Bell.
\newblock An analogue of {C}obham's theorem for fractals.
\newblock {\em Trans. Amer. Math. Soc.}, 2010.
\newblock to appear.

\bibitem[AB10b]{Adamczewski&Bell:xx}
B.~Adamczewski and J.~Bell.
\newblock Automata in number theory.
\newblock In J.-E. Pin, editor, {\em AutoMathA Handbook}. 2010.
\newblock preprint.

\bibitem[AM95]{Allouche&MendesFrance:1995}
J.-P. Allouche and M.~{Mend{\`e}s France}.
\newblock Automata and automatic sequences.
\newblock In {\em Beyond quasicrystals ({L}es {H}ouches, 1994)}, pages
  293--367. Springer, Berlin, 1995.

\bibitem[AS92]{Allouche&Shallit:1992}
J.-P. Allouche and J.~O. Shallit.
\newblock The ring of {$k$}-regular sequences.
\newblock {\em Theoret. Comput. Sci.}, 98:163--197, 1992.

\bibitem[AS03]{Allouche&Shallit:2003}
J.-P. Allouche and J.~O. Shallit.
\newblock {\em Automatic Sequences, Theory, Applications, Generalizations}.
\newblock Cambridge University Press, 2003.

\bibitem[BB07]{Boigelot&Brusten:2007}
B.~Boigelot and J.~Brusten.
\newblock A generalization of {C}obham's theorem to automata over real numbers.
\newblock In {\em Automata, languages and programming}, volume 4596 of {\em
  Lecture Notes in Comput. Sci.}, pages 813--824. Springer, Berlin, 2007.

\bibitem[BBB08]{Boigelot&Brusten&Bruyere:2008}
B.~Boigelot, J.~Brusten, and V.~Bruy{\`e}re.
\newblock On the sets of real numbers recognized by finite automata in multiple
  bases.
\newblock In {\em Proc. 35th ICALP (Reykjavik)}, volume 5126 of {\em Lecture
  Notes in Computer Science}, pages 112--123. Springer-Verlag, 2008.

\bibitem[BBL09]{Boigelot&Brusten&Leroux:2009}
B.~Boigelot, J.~Brusten, and J.~Leroux.
\newblock A generalization of {S}emenov's theorem to automata over real
  numbers.
\newblock In R.~A. Schmidt, editor, {\em Automated Deduction, 22nd
  International Conference, CADE 2009, McGill University, Montreal}, volume
  5663 of {\em Lecture Notes in Computer Science}, pages 469--484, 2009.

\bibitem[{Bel}07]{Bell:2005}
J.~P. {Bell}.
\newblock A generalization of {C}obham's theorem for regular sequences.
\newblock {\em S\'em. Lothar. Combin.}, 54A:Art. B54Ap, 15 pp. (electronic),
  2005/07.

\bibitem[B{\`e}s97]{Bes:1997}
A.~B{\`e}s.
\newblock Undecidable extensions of {B}\"uchi arithmetic and
  {C}obham-{S}em\"enov theorem.
\newblock {\em J. Symbolic Logic}, 62:1280--1296, 1997.

\bibitem[B{\`es}00]{Bes:2000}
A.~B{\`es}.
\newblock An extension of the {C}obham-{S}em{\"e}nov theorem.
\newblock {\em J. Symbolic Logic}, 65:201--211, 2000.

\bibitem[B{\`e}s01]{Bes:2001}
A.~B{\`e}s.
\newblock A survey of arithmetical definability.
\newblock {\em Bull. Belg. Math. Soc. Simon Stevin}, pages 1--54, 2001.
\newblock A tribute to Maurice Boffa.

\bibitem[BHMV94]{Bruyere&Hansel&Michaux&Villemaire:1994}
V.~Bruy\`ere, G.~Hansel, C.~Michaux, and R.~Villemaire.
\newblock Logic and $p$-recognizable sets of integers.
\newblock {\em Bull. Belg. Math. Soc.}, 1:191--238, 1994.

\bibitem[CD08]{Cortez&Durand:2008}
M.~I. Cortez and F.~Durand.
\newblock Self-similar tiling systems, topological factors and stretching
  factors.
\newblock {\em Discrete Comput. Geom.}, 40:622--640, 2008.

\bibitem[{\v C}G86]{Cerny&Gruska:1986}
A.~{\v C}ern\'y and J.~Gruska.
\newblock Modular trellises.
\newblock In G.~Rozenberg and A.~Salomaa, editors, {\em The Book of {L}}, pages
  45--61. Springer-Verlag, 1986.

\bibitem[Chr79]{Christol:1979}
G.~Christol.
\newblock Ensembles presque {p\'eriodiques} $k$-reconnaissables.
\newblock {\em Theoret. Comput. Sci.}, 9:141--145, 1979.

\bibitem[CKMR80]{Christol&Kamae&MendesFrance&Rauzy:1980}
G.~Christol, T.~Kamae, M.~{Mend\`es France}, and G.~Rauzy.
\newblock Suites {alg\'ebriques}, automates et substitutions.
\newblock {\em Bull. Soc. Math. France}, 108:401--419, 1980.

\bibitem[CN03]{Cassaigne&Nicolas:2003}
Julien Cassaigne and Fran{\c{c}}ois Nicolas.
\newblock Quelques propri\'et\'es des mots substitutifs.
\newblock {\em Bull. Belg. Math. Soc. Simon Stevin}, 10:661--676, 2003.

\bibitem[Cob68]{Cobham:1968}
A.~Cobham.
\newblock On the hartmanis-stearns problem for a class of tag machines.
\newblock In {\em IEEE Conference Record of Ninth Annual Symposium on Switching
  and Automata Theory}, pages 51--60, 1968.

\bibitem[Cob69]{Cobham:1969}
A.~Cobham.
\newblock On the base-dependence of sets of numbers recognizable by finite
  automata.
\newblock {\em Math. Systems Theory}, 3:186--192, 1969.

\bibitem[Cob72]{Cobham:1972}
A.~Cobham.
\newblock Uniform tag sequences.
\newblock {\em Math. Systems Theory}, 6:164--192, 1972.

\bibitem[DHS99]{Durand&Host&Skau:1999}
F.~Durand, B.~Host, and C.~Skau.
\newblock Substitutive dynamical systems, {B}ratteli diagrams and dimension
  groups.
\newblock {\em Ergodic Theory Dynam. Systems}, 19:953--993, 1999.

\bibitem[DL06]{Damanik&Lenz:2006}
D.~Damanik and D.~Lenz.
\newblock Substitutional dynamical systems: Characterization of linear
  repetitivity and applications.
\newblock {\em J. Math. Anal. Appl.}, 321:766--780, 2006.

\bibitem[DR09]{Durand&Rigo:2009}
F.~Durand and M.~Rigo.
\newblock Syndeticity and independent substitutions.
\newblock {\em Adv. in Appl. Math.}, 42:1--22, 2009.

\bibitem[DR10]{Durand&Rigo:xx}
F.~Durand and M.~Rigo.
\newblock On {C}obham's theorem.
\newblock preprint, 2010.

\bibitem[Dur96]{Durand:1996}
F.~Durand.
\newblock {\em Contribution \`a l'\'etude des suites substitutives}.
\newblock PhD thesis, Universit\'e de la M\'editerran\'ee, Novembre 1996.

\bibitem[Dur98a]{Durand:1998a}
F.~Durand.
\newblock A characterization of substitutive sequences using return words.
\newblock {\em Discrete Math.}, 179:89--101, 1998.

\bibitem[Dur98b]{Durand:1998b}
F.~Durand.
\newblock A generalization of {C}obham's theorem.
\newblock {\em Theory Comput. Syst.}, 31:169--185, 1998.

\bibitem[Dur98c]{Durand:1998c}
F.~Durand.
\newblock Sur les ensembles d'entiers reconnaissables.
\newblock {\em J. Th{\'e}or. Nombres Bordeaux}, 10:65--84, 1998.

\bibitem[Dur02a]{Durand:2002b}
F.~Durand.
\newblock Combinatorial and dynamical study of substitutions around the theorem
  of {C}obham.
\newblock In {\em Dynamics and randomness ({S}antiago, 2000)}, volume~7 of {\em
  Nonlinear Phenom. Complex Systems}, pages 53--94. Kluwer Acad. Publ.,
  Dordrecht, 2002.

\bibitem[Dur02b]{Durand:2002a}
F.~Durand.
\newblock A theorem of {C}obham for non primitive substitutions.
\newblock {\em Acta Arith.}, 104:225--241, 2002.

\bibitem[Dur08]{Durand:2008}
F.~Durand.
\newblock Cobham-{S}emenov theorem and {$\Bbb N^d$}-subshifts.
\newblock {\em Theoret. Comput. Sci.}, 391:20--38, 2008.

\bibitem[Dur10]{Durand:2010}
F.~Durand.
\newblock Bratteli diagrams.
\newblock In V.~Berth\'e and M.~Rigo, editors, {\em Combinatorics, Automata and
  Number Theory}, volume 135, pages 338--386. Cambridge University Press, 2010.

\bibitem[Eil74]{Eilenberg:1974}
S.~Eilenberg.
\newblock {\em Automata, Languages, and Machines}, volume~A.
\newblock Academic Press, 1974.

\bibitem[Fab94]{Fabre:1994}
S.~Fabre.
\newblock Une {g\'en\'eralisation} du {th\'eor\`eme} de {Cobham}.
\newblock {\em Acta Arith.}, 67:197--208, 1994.

\bibitem[Fab95]{Fabre:1995}
S.~Fabre.
\newblock Substitutions et {$\beta$}-syst\`emes de num\'eration.
\newblock {\em Theoret. Comput. Sci.}, 137:219--236, 1995.

\bibitem[Fag97]{Fagnot:1997}
I.~Fagnot.
\newblock Sur les facteurs des mots automatiques.
\newblock {\em Theoret. Comput. Sci.}, 172:67--89, 1997.

\bibitem[Han82]{Hansel:1982}
G.~Hansel.
\newblock A propos d'un {th\'eor\`eme} de {Cobham}.
\newblock In D.~Perrin, editor, {\em Actes de la {F\^ete} des Mots}, pages
  55--59. Greco de Programmation, CNRS, Rouen, 1982.

\bibitem[HJ90]{Horn&Johnson:1990}
R.~A. Horn and C.~R. Johnson.
\newblock {\em Matrix analysis}.
\newblock Cambridge University Press, Cambridge, 1990.
\newblock Corrected reprint of the 1985 original.

\bibitem[Hon09]{Honkala:2009}
J.~Honkala.
\newblock On the simplification of infinite morphic words.
\newblock {\em Theoret. Comput. Sci.}, 410:997--1000, 2009.

\bibitem[HS03]{Hansel&Safer:2003}
G.~Hansel and T.~Safer.
\newblock Vers un th\'eor\`eme de {C}obham pour les entiers de {G}auss.
\newblock {\em Bull. Belg. Math. Soc. Simon Stevin}, 10:723--735, 2003.

\bibitem[Ked06]{Kedlaya:2006}
K.~S. Kedlaya.
\newblock Finite automata and algebraic extensions of function fields.
\newblock {\em J. Th\'eor. Nombres Bordeaux}, 18:379--420, 2006.

\bibitem[KS75]{Katai&Szabo:1975}
I.~K{\'a}tai and J.~Szab{\'o}.
\newblock Canonical number systems for complex integers.
\newblock {\em Acta Sci. Math. (Szeged)}, 37:255--260, 1975.

\bibitem[LM95]{Lind&Marcus:1995}
D.~Lind and B.~Marcus.
\newblock {\em An introduction to symbolic dynamics and coding}.
\newblock Cambridge University Press, Cambridge, 1995.

\bibitem[Lot83]{Lothaire:1983}
M.~Lothaire.
\newblock {\em Combinatorics on Words}, volume~17 of {\em Encyclopedia of
  Mathematics and Its Applications}.
\newblock Addison-Wesley, 1983.

\bibitem[Muc03]{Muchnik:2003}
A.~Muchnik.
\newblock The definable criterion for definability in {P}resburger arithmetic
  and its applications.
\newblock {\em Theoret. Comput. Sci.}, 290:1433--1444, 2003.

\bibitem[MV93]{Michaux&Villemaire:1993}
C.~Michaux and R.~Villemaire.
\newblock Cobham's theorem seen through {B}\"uchi's theorem.
\newblock In {\em Automata, languages and programming ({L}und, 1993)}, volume
  700 of {\em Lecture Notes in Computer Science}, pages 325--334. Springer,
  Berlin, 1993.

\bibitem[MV96]{Michaux&Villemaire:1996}
C.~Michaux and R.~Villemaire.
\newblock Presburger arithmetic and recognizability of sets of natural numbers
  by automata: {New} proofs of {Cobham's} and {Semenov's} theorems.
\newblock {\em Ann. Pure Appl. Logic}, 77:251--277, 1996.

\bibitem[Pan83]{Pansiot:1982}
J.-J. Pansiot.
\newblock Hi\'erarchie et fermeture de certaines classes de tag-syst\`emes.
\newblock {\em Acta Inform.}, 20:179--196, 1983.

\bibitem[Pan84]{Pansiot:1984a}
J.-J. Pansiot.
\newblock {Complexit\'e} des facteurs des mots infinis {engendr\'es} par
  morphismes {it\'er\'es}.
\newblock In J.~Paredaens, editor, {\em ICALP84}, volume 172 of {\em Lecture
  Notes in Computer Science}, pages 380--389. Springer-Verlag, 1984.

\bibitem[PB97]{Point&Bruyere:1997}
F.~Point and V.~Bruy{\`e}re.
\newblock On the {C}obham-{S}emenov theorem.
\newblock {\em Theory Comput. Syst.}, 30:197--220, 1997.

\bibitem[Per90]{Perrin:1990}
D.~Perrin.
\newblock Finite automata.
\newblock In J.~van Leeuwen, editor, {\em Handbook of Theoretical Computer
  Science, Volume B:Formal Models and Semantics}, pages 1--57. Elsevier --- MIT
  Press, 1990.

\bibitem[Rig00]{Rigo:2000}
M.~Rigo.
\newblock Generalization of automatic sequences for numeration systems on a
  regular language.
\newblock {\em Theoret. Comput. Sci.}, 244:271--281, 2000.

\bibitem[RW06]{Rigo&Waxweiler:2006}
M.~Rigo and L.~Waxweiler.
\newblock A note on syndeticity, recognizable sets and cobham's theorem.
\newblock {\em Bull. European Assoc. Theor. Comput. Sci.}, 88:169--173,
  February 2006.

\bibitem[Sal87]{Salon:1987}
O.~Salon.
\newblock Suites automatiques \`a multi-indices et alg\'ebricit\'e.
\newblock {\em C. R. Acad. Sci. Paris S\'er. I Math.}, 305:501--504, 1987.

\bibitem[Sem77]{Semenov:1977}
A.~L. Semenov.
\newblock The {P}resburger nature of predicates that are regular in two number
  systems.
\newblock {\em Sibirsk. Mat. \v Z.}, 18:403--418, 479, 1977.
\newblock In Russian. English translation in {\it Siberian J.\ Math.} {\bf 18}
  (1977), 289--300.

\bibitem[Sol97]{Solomyak:1997}
B.~Solomyak.
\newblock Dynamics of self-similar tilings.
\newblock {\em Ergodic Theory Dynam. Systems}, 17:695--738, 1997.

\bibitem[SS78]{Salomaa&Soittola:1978}
A.~{Salomaa} and M.~Soittola.
\newblock {\em Automata-theoretic aspects of formal power series}.
\newblock Springer-Verlag, New York, 1978.
\newblock Texts and Monographs in Computer Science.

\end{thebibliography}

\end{document}